\documentclass[a4paper,11pt]{amsart}
\usepackage{amssymb,amsfonts,amsmath}

\def\psh{\mathcal{PSH}}
\def\sh{\mathcal{SH}}
\def\cv{\mathcal{CV}}

\def\C{\mathbb{C}}
\def\c2{\mathbb{C}^2}
\def\Om{\Omega}

\def\1{\bold{1}}

\def\al{\alpha}
\def\be{\beta}
\def\e{\varepsilon}
\def\la{\lambda}

\def\g{\gamma}
\def\et{\eta}
\def\G{\Gamma}

\def\s{\sigma}

\def\pa{\partial}
\def\lbr{\lbrace}
\def\rbr{\rbrace}

\def\we{\wedge}

\newtheorem{lemma}{Lemma}[section]
\newtheorem{pro}[lemma]{Proposition}

\newtheorem{def/not}[lemma]{Definition/Notations}

\newtheorem{thm}[lemma]{Theorem}
\newtheorem{ques}[lemma]{Question}

\newtheorem{remark}[lemma]{Remark}

\newtheorem{exa}[lemma]{Example}

\begin{document}

\title{$m$-subharmonic and $m$-plurisubharmonic functions- on two problems of Sadullaev}

\author{S\l awomir Dinew}
 
\subjclass[2000]{Primary: 32W20. Secondary: 32U15, 32Q15.}

\address{Department of Mathematics and Computer Science \\Jagiellonian University\\ 30-409 Krak\'ow, ul. Lojasiewicza 6, Poland \\
slawomir.dinew@im.uj.edu.pl}

\maketitle
\begin{abstract}  We show that the spaces of $A$-$m$-subharmonic and $B$-$m$-subharmonic functions differ in sufficently high dimensions. We also prove that the Monge-Amp\`ere type operator $\mathcal M_m$ associated to the space of $m$-plurisubharmonic functions does not allow an integral comparison principle except in the classical cases $m=1$ and $m=n$.
These answer in the negative two problems posed by A. Sadullaev.
\end{abstract}
\section*{Introduction} 
Given a domain $\Om\subset\mathbb C^n$ an upper semicontinuous function $u$ defined in $\Omega$ is {\it plurisubharmonic} if for any affine complex line {L} the restriction $u|_{\Om\cap L}$ is subharmonic. This class of functions plays a prominent role in complex analysis- we refer to \cite{De} for some applications.

It is hence natural to try to generalize this class and investigate weaker positivity notions which should be less rigid. Below we mention two natrual generalizations which share many potential theoretic properties with plurisubharmonic functions.

A function is said to be {\it $p$-plurisubharmonic} (or $p$-psh) if it is upper semicontinuous and is subharmonic whenever it is restricted to any affine complex $p$-plane. Thus usual plurisubharmonic functions are $1$-psh, while $n$-psh function
in $\mathbb C^n$ are exactly the subharmonic ones.

When $u$ is additionally $\mathcal C^2$ smooth it is easy to see that $p$-plurisubahrmoni-\newline city can be formulated in either of the following two equivalent ways:
\begin{enumerate}
\item at each point the sum of the $p$-smallest eigenvalues of the complex Hessian of $u$ is nonnegative;
\item  the form $i\pa\bar{\pa}u\wedge(i\pa\bar{\pa}||z||^2)^{p-1}$ is positive.
\end{enumerate}
Such function classes have been previously investigated by Dieu \cite{Di}, Verbitsky \cite{Ve}, Abdullaev \cite{Ab1, Ab2} and Harvey-Lawson- \cite{HL} and appear naturally in various branches of complex analysis- from the regularity of the
Bergman projection (see \cite{HM}) to approximation of $\bar{\partial}$-closed differential forms and Andreotti-Grauert theory -see \cite{Ho, AG, De}.

A related but different class of functions are the $m$-subharmonic ones:

A $\mathcal C^2$ function $u$ is said to be $m$-subharmonic if 
$$(i\pa\bar{\pa}u)^j\wedge (i\pa\bar{\pa}||z||^2)^{n-j}$$
are positve top degree forms for every $j=1,\cdots,m$. Using this positivity and the theory of positive currents it is possible to extend this definition to merely bounded upper semicontinuous functions- see \cite{Bl1,Bl3}.

 In \cite{Sa} A. Sadullaev discussed several aspects of the potential theory associated to $m$-subharmonic and $m$-plurisubharmonic functions. This nice survey covers in particular numerous results of the Uzbekistani complex analysis
 group which are otherwise hardly accessible\footnote{In \cite{Sa} the $m$-psh functions are called $m$-subharmonic, while our definition of $m$-subharmonicity agrees with the notion of B-$n-m+1$-subharmonic functions studied there. 
 As the notion of $m$-subharmonicity provided above is now widely used- see \cite{Bl1,DK1,DK2,DK3} we prefer to stick to this terminology.}.

In the case of $m$-sh functions there is a natural Hessian operator  
$$(i\pa\bar{\pa}u)^m\wedge (i\pa\bar{\pa}||z||^2)^{n-m},$$
which could be defined for all locally bounded $m$-sh functions (see \cite{Bl1,Bl3}) and thus one can recover many analogues of pluripotential theory of Bedford and Taylor- see \cite{Bl1,DK1,DK2,DK3,Sa,AS,Ab2}.

Attempts to build such a theory for $m$-psh functions have been only partially successful (\cite{Sa,Ab1,Ab2}). The basic reason is the lack of a natural Hessian operator associated to this function class. In fact \cite{Sa} lists  two approaches:
\begin{itemize}
\item The first one is to use simply the Hessian $(i\pa\bar{\pa}u)^{n-m+1}\wedge (i\pa\bar{\pa}||z||^2)^{m-1}$. The problem is that this need not return a {\it positive} form as the example of the function $u(z)=-|z_1|^2+|z_2|^2+|z_3|^2$ shows. 
Thus one naturally restricts the class of $m$-psh functions to the set $A-m-sh(\Omega)$ defined by
\begin{equation}\label{amsh}
\lbrace u\in\mathcal C^2(\Omega)|\ u\ {\rm is}\ m-psh, (i\pa\bar{\pa}u)^{n-m+1}\wedge (i\pa\bar{\pa}||z||^2)^{m-1}\geq 0\rbrace.
\end{equation}
\item Alternatively one can seek an operator acting on smooth $m$-psh functions and then try to generalize its action suitably to all locally bounded ones. One possible approach for such an operator is given by
\begin{equation}\label{Mm}
\mathcal M_m(u):=\Pi_{1\leq j_1<j_2<\cdots<j_m\leq n}(\lambda_{j_1}+\lambda_{j_2}+\cdots+\lambda_{j_m}),
\end{equation}
with $\la_j$ denoting the eigenvalues of the complex Hessian of $u$. Obviously this operator is nonnegative on smooth $m$-psh functions and it can be shown that it is elliptic when restricted to this class. It seems however hard to apply pluripotential techniques to $\mathcal M_m$ directly.  We remark nevertheless that $\mathcal M_m$ has also been investigated on manifolds (see \cite{TW}) where it appears naturally in geometric problems.
\end{itemize}

Motivated by these two approaches A. Sadullaev in \cite{Sa} posed the following question:
\begin{ques} Let $u$ be a $(n-m+1)$-subharmonic function (called B-$m$-subharmonic in \cite{Sa}) It can be shown that $u$ is A-$m$-subharmonic. Is the convese true i.e. do we have the equality
$$A-m-sh(\Omega)=B-m-sh(\Omega)?$$
\end{ques}

In \cite{Sa} it is shown that the answer is affirmative for $m=2$, while for $m=1, n-1$ and $n$ the eqauivalence is trivially true. 

A basic tool in pluripotential theory is the {\it integral comparison principle} of Bedford and Taylor- see \cite{BT}. Thus in \cite{Sa} it was asked whether comparsion principle holds for $\mathcal M_m$:
\begin{ques}
Let $u,v\in \mathcal C^2(\Om)\cap \mathcal C(\overline{\Omega})$ be two $m$-psh functions with $u>v$ on $\Om$ with equality on $\partial\Om$. Is it true that for some $\al>0$ and all such tuples $(u,v)$ one has
$$\int_{\Om}\mathcal M_m^\al(u)\leq \int_{\Om}\mathcal M_m^\al(v)?$$
\end{ques}

The goal of this note is to answer in the negative both questions. We will show that the inclusion $B-m-sh(\Omega)\subset A-m-sh(\Omega)$ is strict for domains $\Omega\subset \mathbb C^n$ if $n\geq 11$. Interestingly these function classes are indeed the same 
in dimensions less or equal to 7- see Theorem \ref{1} . As for the second question we will show that this inequality holds only when $\al=1$ and furthermore $m=1$ or $n$ i.e. when we deal with the complex Monge-Amp\`ere operator or the Laplacian.

In Author's opinion these negative results show that a construction of potential theory for $m$-psh function is necessarily subtler than in the case of $m$-subharmonic ones and attempts to apply directly tools from Bedford-Taylor theory are doomed to fail. It seems however
possible that a suitable {\it viscosity} potential theory can be constructed- see \cite{DDT} for such an approach.

{\bf Acknowledgements} The Author is grateful to the organizers of the 2nd USA-Uzbekistan conference {\it Mathematics and Mathematical Physics} held in Urgench in 2017, where in particular he learned about these problems and to professor A. Sadullaev for helpful comments. 
The Author was supported by Polish National Science Centre grant 2017/26/E/ST1/00955. 
\section{Preliminaries}
In this section we shall fix our terminolory. Throughout the note we shall work with $\mathcal C^2$ functions hence all operators involved will have a {\it classical} meaning. We refer to \cite{Bl1,DK1,DK2,DK3} for the nonlinear potetnial theory of weak $m$-subharmonic functions.

 Consider the set $\mathcal
A_n$ of all Hermitian symmetric $n\times n$ matrices. For a given
matrix $M\in\mathcal A_n$ let $ \la(M) = (\la _1, \la _2 , ... ,
\la _n)$ be its eigenvalues arranged in increasing order and
let
$$
\s_k(M)=\s_k(\la(M))=\sum _{0<j_1 < ... < j_k \leq n }\la _{j_1}\la _{j_2} 
... \la _{j_k}
$$
be the $k$-th elementary symmetric polynomial applied to the vector $\la(M)$. We shall simply write $\la$ and $\s_k(\la)$
if the matrix $M$ in question is clear from the context. Also we shall use the convention $\s_0(\la):=1$ and $\s_j(\la)=0$
if $\la$ is a vector of less than $j$ coordinates.

We denote by $\s_j(\la|\la_{i_1},\cdots,\la_{i_r})$ the value of $\s_j$ when the coefficients $\la_{i_m}$ are exchanged by zero. Alternatively 
this is the $j$-th elementary symmetric polynomial on the remaining coefficients.

Denote by $S_k(\la):=\frac{\s_k(\la)}{\binom nm}$ the normalized Hessian operators. The normalization is chosen so that
$S_k(t\1)=t^k$ if $\1$ denotes the vector with all coefficients equal to one.

Then one can define the positive cones $\Gamma_m$ as follows
\begin{equation}\label{ga}
\Gamma_m=\lbrace \la\in\mathbb R^n|\ S_1(\la)> 0,\ \cdots,\
S_m(\la)> 0\rbrace.
\end{equation}
Note that the definition of $\Gamma_m$ is non linear if $m>1$.

Below we list the properties of these cones that will be used later on.
\begin{pro}[Maclaurin's inequality]\label{McL1}
If $\la\in\Gamma_m$ then
$$({S_j}(\la))^{\frac 1j}\geq ({S_i}(\la))^{\frac 1i}$$
for $1\leq j\leq i\leq m$. 
\end{pro}
\begin{pro}[Newton inequality]\label{N1}
Let $\la\in\mathbb R^n$ be any vector. Then for any $k\in\{1,2,\cdots,n-1\}$ one has
$$S_{k-1}(\la)S_{k+1}(\la)\leq S_k^2(\la).$$
\end{pro}
We emphasize that the inequality holds for {\it} any vector and not only for those belonging to cones $\G_j$. 

If all the $\s_j$'s
are positive it is easy to derive a slightly weaker inequality:
\begin{pro}[weak Newton inequality]\label{N2}
If $\la\in\G_k$  then for any $j\in\{1,2,\cdots,k-1\}$ one has
$$\s_{j-1}(\la)\s_{j+1}(\la)\leq \s_j^2(\la).$$
\end{pro}
\begin{proof}
 Newton inequality in terms of $\s_j$'s is simply
 $$\s_j(\la)^2\geq \s_{j-1}(\la)\s_{j+1}(\la)\frac{(n-j+1)(j+1)}{(n-j)j}.$$
 It remains to observe that the last constant is larger than $1$.
\end{proof}

The next proposition is a classical result in vector analysis:
\begin{pro}\label{vector} If the vector $\beta$ belongs to $\Gamma_k$, then the sum of any $n-k+1$-coeffictients of $\beta$ is non-negative. In particular any $\mathcal C^2$ smooth $m$-subharmonic function is $n-m+1$-plurisubharmonic.
\end{pro}

The following summation formula is easy to prove:
\begin{pro}[Summation formula]\label{sumf}
 For a vector $\g\in\mathbb R^{n-1}$ let $\al$ denotes the vector in $\mathbb R^p$ formed by the first $p$-coordinates of $\g$, while
 $\be$- the vector formed by the remaining coordinates. Then
 $$\sigma_{j}(\g)=\sum_{i=0}^j\s_i(\al)\s_{j-i}(\be).$$
\end{pro}

We refer to \cite{Bl1} or \cite{Wa2} for further properties of these cones.

Recall that the operator $\mathcal M_m$ is given by
$$\mathcal M_m(u):=\Pi_{1\leq j_1<j_2<\cdots<j_m\leq n}(\lambda_{j_1}+\lambda_{j_2}+\cdots+\lambda_{j_m}).$$
As $\mathcal M_u$ is defined through a symmetric polynomial of the eigenvalues of $\frac{\pa^2 u}{\pa z_j\pa\bar{z}_k}(z)$ it follows from the fundamental theorem of symmetric polynomials that it can be expressed through
$\sigma_p(\lambda)=\s_p(\frac{\pa^2 u}{\pa z_j\pa\bar{z}_k}(z))$, $p=1,\cdots,m$. One observes that $\mathcal M_1(u)=\sigma_n(\frac{\pa^2 u}{\pa z_j\pa\bar{z}_k}(z))$ is the complex Monge-Amp\`ere operator,
$\mathcal M_n(u)=\s_1(\frac{\pa^2 u}{\pa z_j\pa\bar{z}_k}(z))$ is simply the Laplacian. The expression of concrete $n$ and $m$ can be complicated- in particular for $m=2$ and $n=3$ it can be computed that $\mathcal M_2(u)=\s_1(u)\s_2(u)-\s_3(u)$.

\section{A-$m$-subharmonicity versus B-$m$-subharmonicity}

Recall that a (smooth) function $u$ is $m$-subharmonic if $\s_j(\frac{\pa^2u}{\pa z_j\pa\bar{z}_k}(z))\geq0$ for every $j=1,\cdots,m$ and every point in the domain of definition of $u$. These are called B-$(n-m+1)$-subharmonic in \cite{Sa}, a terminology that we shall apply in this section.

A smooth function is $A$-m-sh if it is $m$-plurisubharmonic and satisfies $\s_{n-m+1}(\frac{\pa^2u}{\pa z_j\pa\bar{z}_k}(z))\geq0$.

Note that $B-m-sh\subset A-m-sh$ thanks to Proposition \ref{vector} and if there were an equality that would mean that checking $m$-subharmonicity reduces to checking that the $m$-Hessian is positive (a thing which in potential theory is usually given a priori) and furthermore
$$i\pa\bar{\pa}u\wedge( i\pa\bar{\pa}||z||^2)^{n-m}\geq 0$$ 
which is a linear condition.

In \cite{Sa} it was shown that  if $m=2$ then for every $n$ both notions indeed coincide. More generally they coincide for functions with at most one non-positive eigenvalue.

In this section we solve Sadullaev's problem. More precisely we prove that in a domain $\Om\subset\mathbb C^n$ Blocki's notion 
of $m$-subharmonic functions agrees with the one of Abdullaev provided that $n\leq 7$. We also show that this fails in large dimensions.
\begin{thm}\label{1}
Let $u\in\mathcal C^2(\Om)$, where $\Om\subset\mathbb C^n,\ n\leq 7$. Then $u$ is $n-k+1$-subharmonic (or B-$k$-subharmonic) if and only if it satisfies
$$i\pa\bar{\pa}u\we( i\pa\bar{\pa}||z||^2)^{k-1}\geq 0,\ \ \ (i\pa\bar{\pa}u)^{n-k+1}\we( i\pa\bar{\pa}||z||^2)^{k-1}\geq 0.$$
\end{thm}
\begin{proof}
 If $u$ is $(n-k+1)$-subharmonic then it satisfies Abdullaev's conditions by Proposition \ref{vector} (see also \cite{Sa}). In order to prove
 the reverse implication we argue at a fixed point $z_0\in\C^n$. By a complex linear change of coordinates we can assume that
 the complex Hessian of $u$ is diagonal at $z_0$.

 Observe first that the case $k=n-1$ is trivial and the case $k=2$ was done in \cite{Sa}.

 If all the eigenvalues are non negative they obviously form a vector in $\G_{n-k+1}$ and there is nothing to prove. Thus we suppose that there
 is a negative smallest eigenvalue (called $\alpha_0$) which, afer scaling if necessary we assume to be equal to $-1$. 
 Let $-1=\al_0\leq\al_1\leq\cdots\leq \al_p<0$ denote all the negative eigenvalues. If $p=0$ then the proof from
 \cite{Sa} works, hence we assume $p\geq1$ in what follows. Similarly let
 $0\leq \be_1\leq\cdots\leq\be_{n-1-p}$ denote the nonnegative eigenvalues.
 
 We denote by $\g$ the vector $\g:=(\al,\be)$ with first $p$-coordinates equal to $\al_j, j=1,\cdots, p$ and last $(n-p-1)$-
 coordinates equal to $\be_j, j=1,\cdots, n-p-1$. Similarly we define $\et:=(-1,\g)$.
 
 Our goal is to prove that for $n\leq 7$
 \begin{equation}\label{eta}
  \et\in\G_{n-k+1}\ \ {\rm i.e.}\ \ \s_j(\et)\geq0,\ \ j=1,\cdots,n-k-1.
 \end{equation}
Note that $(i\pa\bar{\pa}u)^{n-k+1}\we(i\pa\bar{\pa}||z||^2)^{k-1}\geq 0$ at $z_0$ can be rewritten in the language of eigenvalues as
\begin{equation}\label{n-k+1}
 0\leq\s_{n-k+1}(\et)=\s_{n-k+1}(\et|\et_1)+\et_1\s_{n-k}(\et|\et_1)=\s_{n-k+1}(\g)-\s_{n-k}(\g).
\end{equation}

Note that for any $j=1,\cdots, n-k-1$
$$\s_j(\et)=\s_j(\g)-\s_{j-1}(\g)$$ 
thus it suffices to prove
\begin{equation}\label{g}
{\rm For\ every}\ j\in 1,\cdots, n-k-1\ {\rm the\  inequality}\ \s_{j}(\g)-\s_{j-1}(\g)\geq 0\ {\rm holds.}
\end{equation}

The condition $(i\pa\bar{\pa}u)\we(i\pa\bar{\pa}||z||^2)^{k-1}\geq 0$ means that the sum of any $k$-tuple of eigenvalues is nonnegative. 
Note that this in particular implies
that $p\leq k-2$.

We claim that it suffices to prove that $\s_j(\g)> 0, j=1,\cdots,n-k$ i.e. $\g\in\G_{n-k}$. Indeed suppose this were true.

Then from (\ref{n-k+1}) $\s_{n-k+1}(\g)\geq\s_{n-k}(\g)\geq 0$, and from Proposition \ref{N2}  we have
$$\s_{n-k}(\g)^2\geq \s_{n-k+1}(\g)\s_{n-k-1}(\g)\geq\s_{n-k}(\g)\s_{n-k-1}(\g).$$
Exploiting the positivity once again we end up with
$$\s_{n-k}(\g)\geq \s_{n-k-1}(\g).$$
Repeating the argument we obtain $\s_{n-k-1}(\g)\geq \s_{n-k-2}(\g)$ and so on up until $\s_{1}(\g)\ge \s_{0}(\g)$.

Let us proceed with the proof of the claim.

As $i\pa\bar{\pa}u\we(i\pa\bar{\pa}||z||^2)^{k-1}\geq 0$ we obtain that $i\pa\bar{\pa}u\we(i\pa\bar{\pa}||z||^2)^{n-1}\geq 0$ i.e. $\s_1(\g)\geq 1$ (recall that $\et_1=-1$).

From Proposition \ref{sumf} we know that

$$\s_j(\g):=\sum_{i=0}^{min\{j,p\}}\s_{j-i}(\be)\s_{i}(\al).$$

Recall that the case $k=n-1$ is trivial. Thus we assume from now on that $k\leq n-2\leq 5$ and hence $p\leq 3$.
In fact $min\{j,p\}\leq 2$- if
$p\geq3$ then $n=7, k=5$ which yields $j\leq 2$. Thus
$$\s_j(\g)=\s_j(\be)+\s_1(\al)\s_{j-1}(\be)+\s_2(\al)\s_{j-2}(\be),$$
where the last term is assumed to be zero if $j=1$ or $p=1$. Observe that, whenever defined, this last term is always non negative
hence we have the fundamental inequality
\begin{equation}\label{fundeq}
\s_j(\g)\geq\s_j(\be)+\s_1(\al)\s_{j-1}(\be).
\end{equation}

Observe that
\begin{equation}\label{divisionofb}
 \s_j(\be)=\frac1j(\sum_{1\leq l_1<\cdots<l_{j-1}\leq n-p-1}\be_{l_1}\cdots \be_{l_{j-1}}
 [\sum_{l\notin\{l_1,\cdots l_{j-1}\}}\be_l]).
\end{equation}
On the other hand for any $(k-p-1)$-tuple $1\leq r_1<\cdots r_{k-p-1}\leq n-p-1$
the sum $-1+\s_1(\al)+\sum_{s=1}^{k-p-1}\be_{r_s}$ is nonnegative by assumption. Summing over all $(k-p-1)$-tuples such that
$$\{r_1,\cdots,r_{k-p-1}\}\cap \{l_1,\cdots,l_j\}=\varnothing$$
we obtain

$$\binom{n-p-j}{k-p-1}(-1+\s_1(\al))+\binom{n-p-j-1}{k-p-2}\sum_{l\notin\{l_1,\cdots l_{j-1}\}}\be_l\geq 0.$$
Coupling this with the elementary inequlity $\s_1(\al)\geq -p$ (since all $\al_j\geq -1$) we obtain
\begin{equation}\label{helpful}
\sum_{l\notin\{l_1,\cdots l_{j-1}\}}\be_l\geq \frac{(n-p-j)(p+1)}{(k-p-1)p}(-\s_1(\al)). 
\end{equation}
Summing over in equation (\ref{divisionofb}) we get
$$\s_j(\be)\geq\s_{j-1}(\be)(-\s_1(\al))\frac{(n-p-j)(p+1)}{j(k-p-1)p}$$

Thus the fundamental inequality (\ref{fundeq}) yields
\begin{equation}\label{bin}
 \s_j(\g)\geq \s_{j-1}(\be)(-\s_1(\al))[\frac{(n-p-j)(p+1)}{j(k-p-1)p}-1].
\end{equation}
Hence $\s_j(\g)\geq 0$ provided 
$$\frac{(n-p-j)(p+1)}{j(k-p-1)p}\geq 1.$$
The quantity on the left is clearly decreasing in $j$, hence it is smallest for $j=n-k$ and then reads
$$\frac{(k-p)(p+1)}{(n-k)(k-p-1)p}.$$
It is then straightforward to check that the latter quantity is indeed at least $1$ for all triples $(p,k,n)\in\mathbb N^3$ such that
$1\leq p\leq k-2, k\leq n-2, n\leq 7$.

\end{proof}
The following example shows that Blocki's and Abdullaev's notions are different in large dimensions even for $k=3$:
\begin{exa}
 Consider the function 
 $$u(z_1,\cdots,z_{11})=-\sum_{j=1}^2|z_j|^2+2\sum_{j=3}^{11}|z_j|^2.$$
 Then $(i\pa\bar{\pa}u)\we(i\pa\bar{\pa}||z||^2)^{2}\geq 0, (i\pa\bar{\pa}u)^{9}\we(i\pa\bar{\pa}||z||^2)^{2}\geq 0$, but $(i\pa\bar{\pa}u)^{8}\we(i\pa\bar{\pa}||z||^2)^{3}<0$. i.e. $u$ is $m$-sh in the sense of Abdullaev but 
 not in the sense of Blocki.
\end{exa}
\begin{proof}
 By computation 
 $$\s_9(\la(\frac{\partial^2 u}{\partial z_j\partial\bar{z}_k}))(z)=\s_2(-1,-1,2,2,\cdots,2)=2^9-2.9.2^8+36.2^7>0$$
 but
 $$\s_8(-1,-1,2,\cdots,2)=9.2^8-2.36.2^7+84.2^6=-24.2^6,$$
 as claimed.
\end{proof}

\begin{section}{Failure of the integral comparison principle}
Recall that an elliptic operator $F(\frac{\pa^2u}{\pa z_p\pa \bar{z}_q})$ is said to satisfy the integral comparison principle if for any two $\mathcal C^2$ admissible functions $u$ and $v$ defined in a domain $\Omega\subset\mathbb C^n$ one has

$$\int_{\lbr u<v\rbr}F(\frac{\pa^2v}{\pa z_p\pa \bar{z}_q})\leq \int_{\lbr u<v\rbr}F(\frac{\pa^2v}{\pa z_p\pa \bar{z}_q})$$
provided $u\geq v$ on $\pa\Omega$. Classical examples include the Laplacian and the complex Monge-Amp\`ere operator restricted to the class of plurisubharmonic functions. Recall that in \cite{BT} the validity of the comparison principle has been extended to all locally bounded plurisubharmonic functions. 

Such an inequality would have been very helpful in developing a version pluripotential theory associated to $\mathcal M_m$. On the other hand if one wants Chern-Levine-Nirenberg inequalities to hold (see \cite{BT})- which is again a basic property in pluripotential theory allowing in particular to define relative capacities, it is more natural to consider the operator $\mathcal M_m^{\al}$ with the exponent $\alpha$ chosen properly.

Unfortunately our next result shows that this is impossible (for every choice of $\alpha$) unless $p=1$ or $p=n$:
 \begin{thm}\label{compprin} Suppose that the operator $\mathcal M_m^{\al}(u)$ satisfies the integral comparison principle. Then $\al=1$ and furthermore $m=1$ or $m=n$.
\end{thm}
 Before starting the proof  we need a lemma that is a minor generalization of Lemma 1.2 in \cite{CNS}:
\begin{lemma}\label{linalg}
Consider the real valued function
$$\rho(z):=\chi(z)+\sum_{j=1}^na_j|z_j|^2,$$
where $\chi$ is any $C^2$ function  $a_j>0, a_1<a_2<\cdots<a_n$ and $a_j\leq C$ for $j=1,\cdots, n-1$ for some constant $C$. Then, assuming that $a_n$ tends to infinity, 
$(\frac{\pa^2\rho}{\pa z_p\pa \bar{z}_q})_{p,q=1,\cdots, n}$ has eigenvalues  $\lambda_j((\frac{\pa^2\rho}{\pa z_p\pa \bar{z}_q})_{p,q=1}^n)$, $j=1,\cdots,n$ at a fixed point $z$ satisfying
$$\lambda_j((\frac{\pa^2\rho}{\pa z_p\pa \bar{z}_q})_{p,q=1}^n)=\tilde{\lambda}_j((\frac{\pa^2\rho}{\pa z_p\pa \bar{z}_q})_{p,q=1}^{n-1})+o(1)$$
 for $1\leq j\leq  n-1$, while 
$$\lambda_n=a_{n}+\chi_{n\bar{n}}+o(1).$$ All $o(1)$ terms are uniform and depend on $C$ and the $C^2$ bound on $\chi$.
\end{lemma}
\begin{proof}
Dividing the last row of the characteristic equation
$$det(\frac{\pa^2\rho}{\pa z_p\pa \bar{z}_q}-t I_n)=0$$
by $a_n$ and then passing with $a_n$ to infinity we obtain that the all but one of the roots  satisfy the characteristic equation for the matrix
$(\frac{\pa^2\rho}{\pa z_p\pa \bar{z}_q})_{p,q=1}^{n-1}$ and the first part of the claim follows from the continuous dependence of eigenvalues with respect to the matrix coefficients. The equality of $\lambda_n$ follows simply from taking the traces of $(\frac{\pa^2\rho}{\pa z_p\pa \bar{z}_q})_{p,q=1}^n$ and $(\frac{\pa^2\rho}{\pa z_p\pa \bar{z}_q})_{p,q=1}^{n-1}$.

\end{proof}
\begin{proof}[Proof of Theorem \ref{compprin}]

We begin with the following claim providing a lower bound for $\alpha$:

{\bf Claim 1:}Let $\Omega$ be a bounded domain with $\mathcal C^2$ smooth boundary. If for any $m$-psh functions $u,v\in \mathcal C^2(\overline{\Omega})$ satisfying   $u\geq v$ in $\Omega$, $u=v$ on $\partial\Omega$, one has
$$\int_{\Omega}\mathcal M_m^{\al}(u)\leq\int_{\Omega}\mathcal M_m^{\al}(v),$$
for some $\al>0$, then $\al\geq 1$.

Fix any strictly negative smooth function $\chi$ on the unit ball $B_1(0)$ which vanishes together with its gradient on $\pa B_1(0)$. As a concrete example we may take $\chi(z):=-(1-||z||^2)^2$. Take $u(z):=\sum_{j=1}^na_j|z_j|^2$,
$v(z):=\sum_{j=1}^na_j|z_j|^2+\chi$, for some sufficiently large positive  constants $a_j$, so that both $u$ and $v$ are $m$-psh. If the integral comparison principle were true then
\begin{equation}\label{beforedivision}
\int_{B_1(0)}\mathcal M_m^{\al}(u)\leq\int_{B_1(0)}\mathcal M_m^{\al}(v).
\end{equation}

Let now $a_n$ to infinity, while keeping other $a_j$'s fixed. Using lemma \ref{linalg} in Equation (\ref{beforedivision}), after dividing both sides by $a_n^{\binom{n-1}{m-1}\al}$ we end up with
$$\int_{B_1(0)}\tilde{\mathcal M}_m^{\al}(u)\leq\int_{B_1(0)}\tilde{\mathcal M}_m^{\al}(v),$$
where the $\tilde{}$ sign denotes computation of $\mathcal M_m$ in the first $n-1$-coordinates. Letting now $a_{n-1}$ to infinity we can repeat the process. After the $n-m$-th iteration it is easy to see that we end up with
\begin{equation}\label{afterdivision}\int_{B_1(0)}(\sum_{j=1}^ma_j)^{\al}\leq\int_{B_1(0)}(\sum_{j=1}^ma_j+\chi_{j\bar{j}}(z))^{\al}.
\end{equation}

Note that inequality (\ref{afterdivision}) holds for {\it all} $\mathcal C^2$ smooth functions $\chi$ assuming that $a_j$'s, $j=1,\cdots,m$ are large enough.  In particular taking the path $v_{\e}:=\sum_{j=1}^na_j|z_j|^2+\e\chi$ for $a_1,\cdots a_m$ fixed and applying the whole process we end up with

$$\int_{B_1(0)}(\sum_{j=1}^ma_j)^{\al}\leq\int_{B_1(0)}(\sum_{j=1}^ma_j+\e\chi_{j\bar{j}}(z))^{\al}.$$ 

Expanding the right hand side in $\e$ we obtain
\begin{align*}
&\int_{B_1(0)}(\sum_{j=1}^ma_j+\e\chi_{j\bar{j}}(z))^{\al}=\int_{B_1(0)}(\sum_{j=1}^ma_j)^{\al}\\
&+\int_{B_1(0)}\al(\sum_{j=1}^ma_j)^{\al-1}\e\sum_{k=1}^m\chi_{k\bar{k}}(z)\\
&+\int_{B_1(0)}\al(\al-1)(\sum_{j=1}^ma_j)^{\al-2}\frac{\e^2}2(\sum_{k=1}^m\chi_{k\bar{k}}(z))^2+o(\e^2)\\
&=I+II+III+IV.
\end{align*}
The first term clearly matches the left hand side in (\ref{afterdivision}). The second term is zero as it can be seen from integration by parts (we use the vanishing of the gradient of $\chi$ at the boundary). But the third term is strictly negative if $\al<1$ and it dominates the fourth one for small $\e$. Hence we must have $\al\geq 1$, which yields the claim.

The next claim in turn provides an upper bound for $\alpha$:

{\bf Claim 2:} Let $B_1(0)$ be the unit ball in $\C^n$. If for any rotationally invariant $m$-psh functions $u,v\in \mathcal C^2(\overline{B_1(0)})$ satisfying   $u\geq v$ in $\Omega$, $u=v$ on $\partial\Omega$, one has
$$\int_{B_1(0)}\mathcal M_m^{\al}(u)\leq\int_{B_1(0)}\mathcal M_m^{\al}(v),$$
for some $\al>0$, then $\al\leq \frac{1}{\binom{n-1}{m-1}}$.

It is straightforward to compute that if $u(z):=\chi(||z||^2)$ for some $\mathcal C^2$ smooth function $\chi$, then
the eigenvalues of the complex Hessian of $u$ satisfy 
$$\lambda_1(z)=\cdots=\lambda_{n-1}(z)=\chi'(||z||^2),\ \lambda_n(z)=\chi'(||z||^2)+||z||^2\chi''(||z||^2).$$
Thus
\begin{equation}\label{radialmu}
 \mathcal M_m(u)(z)=[(m\chi'(||z||^2))^{\binom{n-1}{m}}(m\chi'(||z||^2))+||z||^2\chi''(||z||^2))^{\binom{n-1}{m-1}}].
\end{equation}
We apply this for the family of $\chi_A(t):=\frac12(\frac{(t+A)^2}{1+A}-1+A),\ A\geq 0$. It is easy to see that the
corresponding functions $u_A(z):=\chi_A(||z||^2)$ are plurisubharmonic, hence $m$-psh, and they all vanish on the unit sphere.
Also $u_A$ is a decreasing sequence as $A$ increases.

Supposing that the comparison principle holds for some $\al$ we obtain then
$$\int_{B_1(0)}\mathcal M_m^{\al}(u_0)\leq\int_{B_1(0)}\mathcal M_m^{\al}(u_A)$$
for any $A>0$. But then, denoting by $c_{2n-1}$ the area of the unit sphere, the left hand side is simply
\begin{align*}
&\int_{B_1(0)}M_m(u_0)^{\al}=\int_{B_1(0)}[(m\chi')^{\binom{n-1}{m}}
(m\chi'+||z||^2\chi'')^{\binom{n-1}{m-1}}]^\al\\
&=\int_{B_1(0)}[(m||z||^2)^{\binom{n-1}{m}}((m+1)||z||^2)^{\binom{n-1}{m-1}}]^{\al}\\
&=c_{2n-1}m^{\binom{n-1}{m}\al}(m+1)^{\binom{n-1}{m-1}\al}\int_0^1r^{2n-1+2\binom{n-1}{m-1}\al+2\binom{n-1}{m}\al} dr\\
&=c_{2n-1}\frac{m^{\binom{n-1}{m}\al}(m+1)^{\binom{n-1}{m-1}\al}}{2n+2\binom nm\al}.
\end{align*}

On the other hand after taking the limit as $A\rightarrow\infty$ the right hand side becomes

\begin{align*}
 lim_{A\rightarrow\infty}\int_{B_1(0)}M_m(u_A)^{\al}=&\int_{B_1(0)}m^{\binom nm\al}=c_{2n-1}m^{\binom nm\al}\int_0^1r^{2n-1}dr\\
 &=c_{2n-1}\frac{m^{\binom nm\al}}{2n}.
\end{align*}
Comparing both sides we obtain the numerical inequality
\begin{equation}\label{outcome}
 c_{2n-1}\frac{m^{\binom{n-1}{m}\al}(m+1)^{\binom{n-1}{m-1}\al}}{2n+2\binom nm\al}\leq c_{2n-1}\frac{m^{\binom nm\al}}{2n},
\end{equation}
which reduces to
$$(1+\frac1m)^{\binom{n-1}{m-1}\al}\leq 1+\frac1m\binom{n-1}{m-1}\al.$$
If now $\al>\frac1{\binom{n-1}{m-1}}$ we get the contradiction with the elementary 
inequality $(1+x)^\beta>1+\beta x$, valid for all $x>0$ and $\beta>1$.

Finally coupling Claim 1 with Claim 2 it is obvious that the comparison principle can hold iff $\al=1$ and $\binom{n-1}{m-1}=1$ i.e.
if $m=1$- the case of the complex Monge-Amp\`ere operator or $m=n$- the Laplacian.

\end{proof}
\begin{remark}
It is interesting to note that for {\it radial} $m$-psh functions the comparison principle does hold true for the operator $\mathcal M_m$ raised to
power $\frac{1}{\binom{n-1}{m-1}}$. We leave the elementary proof to the Reader.
\end{remark}
\end{section}


\begin{thebibliography}{CKNS86}
\bibitem[Ab1]{Ab1}B. Abdullaev, Subharmonic functions on complex Hypersurfaces in $\mathbb C^n$, \it J. Siberian Fed. Univ., Mathematics and Physics\rm\ \ 6 (4) (2013), 409-416.
\bibitem[Ab2]{Ab2} B. Abdullaev, $\mathcal P$- measure in the class of $m$-wsh functions, \it J. Siberian Fed. Univ., Mathematics and Physics\rm\ \ 7 (1) (2014), 3-9.
\bibitem[AG]{AG} A.Andreotti, H.Grauert, Th\'eor\`emes de finitude
 pour la cohomologie des espaces complexes, \it Bull. Soc. Math.
 France \rm 90 (1962) 193-259.
\bibitem[AS]{AS} B. Abdullaev B, A. Sadullaev, Potential theory in the class of $m$-sh functions.
\it Proc. Steklov Inst. Math. \rm 279 (2012), 155-180. 
\bibitem[BT]{BT} E. Bedford, B. A. Taylor, A New Capacity for Plurisubharmonic Functions,\it Acta Math. \rm (149) (1982), 1-40.
 \bibitem[Bl1]{Bl1} Z. B\l ocki, Weak solutions to the complex Hessian
 equation, \it Ann. Inst. Fourier (Grenoble) \rm\ (55), 5 (2005),
 1735-1756.
 \bibitem[Bl2]{Bl3} Z. B\l ocki, Defining nonlinear elliptic operators for
non-smooth functions, Complex Analysis and Digital Geometry proceedings, Uppsala (2009).

 \bibitem[CNS]{CNS}  Caffarelli, L.; Nirenberg, L.; Spruck, J. The Dirichlet problem for nonlinear second-order
 elliptic equations. III. Functions of the eigenvalues of the Hessian. \it Acta Math. \rm\ (155) (1985), no. 3-4, 261-301.
 \bibitem[De]{De} J,-P.Demailly, Complex Analytic and Differential
 Geometry (1997), can be found at the www site
 http://www-fourier.ujf-grenoble.fr/~demailly/books.html
 \bibitem[Di]{Di} N. Q. Dieu, q-plurisubharmonicity and q-pseudoconvexity in $\mathbb C^n$, \it Publ. Mat. \rm\ (50) (2006), 349-369.
 
  \bibitem[DDT]{DDT} S. Dinew, H. S. Do, T. D. To, A viscosity approach to the Dirichlet problem for degenerate complex
Hessian type equations, \it Anal. PDE\rm\ \ (12) (2019), no. 2, 505-535.

\bibitem[DK1]{DK1} S. Dinew, S. Ko\l odziej, A priori estimates for complex Hessian equations. \it Anal. PDE \rm\ (7) (2014), no. 1, 227-244. 

\bibitem[DK2]{DK2} S. Dinew, S. Ko\l odziej, Liouville and Calabi-Yau type theorems for complex Hessian equations, \it
 Amer. J. Math. \rm\ (139) (2017), no. 2, 403-415.
 \bibitem[DK3]{DK3} S. Dinew, S. Ko\l odziej, Non standard properties of $m$-subahrmonic functions, \it Dolom. Res. Not. Approx.\rm\ (11) (2018), 35-50.
 
\bibitem[HL]{HL} F. R. Harvey, H. B. Lawson, $p$-convexity, $p$-plurisubharmonicity and the Levi problem. \it Indiana Univ. Math. J.\rm\ (62) (2013), no. 1, 149-169. 
 
 \bibitem[HM]{HM} A.-K. Herbig, J. D. McNeal, Regularity of the Bergman projection on forms and plurisubharmonicity conditions,
 \it Math. Ann.\rm\ (336) (2006) 335-359.
 \bibitem[Ho]{Ho} L.-H. Ho, $\bar{\partial}$-problem on weakly $q$-convex domains. \it Math. Ann.\rm\ (290) (1991), no. 1, 3-18.
 
  \bibitem[Sa]{Sa} A. Sadullaev, Further developments of the pluripotential theory (survey). \it Algebra, complex analysis, and pluripotential theory,\rm 167-182,
Springer Proc. Math. Stat., (264), Springer, Cham, (2018).

\bibitem[TW]{TW} V. Tosatti, B. Weinkove, The Monge-Amp\`ere equation for $(n-1)$-plurisubharmonic functions on a compact K\"ahler manifold,
\it J. Amer. Math. Soc. \rm\ (30) (2017), no. 2, 311-346.

\bibitem[Ve]{Ve} M. Verbitsky, Plurisubharmonic functions in calibrated geometry and $q$-convexity. \it Math. Z.\rm\ (264) (2010), no. 4, 939-957.
  \bibitem[W1]{Wa2} X.-J. Wang, The $k$-Hessian equation, \it Lect. Not. Math.
\rm 1977 (2009).
\end{thebibliography}
\end{document}